\documentclass[12pt]{article}
\usepackage[a4paper,margin=1.0in]{geometry}
\usepackage[colorlinks,citecolor=magenta,linkcolor=black]{hyperref}
\pdfpagewidth=\paperwidth \pdfpageheight=\paperheight
\usepackage{amsfonts,amssymb,amsthm,amsmath,eucal,tabu,url,nccmath,empheq}
\usepackage{systeme}
\usepackage{fancyvrb}
\usepackage{pgf}
 \usepackage{array}
 \usepackage{tikz-cd}
 \usepackage{pstricks}
 \usepackage{pstricks-add}
 \usepackage{pgf,tikz}
 \usetikzlibrary{automata}
 \usetikzlibrary{arrows}
 \usepackage{indentfirst}
\usepackage{tabularx} 
\usepackage{booktabs}
\usepackage{xcolor}
\usepackage{longtable}

\theoremstyle{plain}
\newtheorem{thm}{Theorem}[section]
\newtheorem{theorem}[thm]{Theorem}

\newtheorem{proposition}[thm]{Proposition}

\theoremstyle{definition}
\newtheorem{definition}[thm]{Definition}
\newtheorem{remark}[thm]{Remark}

\newtheorem{thevarthm}[thm]{\varthmname}

\newenvironment{varthm*}[1]{\trivlist\item[]{\bf #1.}\it}{\endtrivlist}

\renewcommand\geq{\geqslant}

\renewcommand\leq{\leqslant}

\newcommand\be{\begin{eqnarray*}}
\newcommand\ee{\end{eqnarray*}}

\newcommand\newop[2]{\def#1{\mathop{\rm #2}\nolimits}}
\newop\log{log}
\newop\ord{ord}
\newop\Gal{Gal}
\newop\SL{SL}
\newop\Bl{Bl}
\newop\mult{mult}
\newop\mass{mass}
\newop\div{div}
\newop\codim{codim}
\newop\sing{sing}
\newop\vdim{vdim}
\newop\edim{edim}
\newop\Ass{Ass}
\newop\size{size}
\newop\reg{reg}
\newop\satdeg{satdeg}
\newop\supp{supp}
\newop\Neg{Neg}
\newop\Nef{Nef}
\newop\Nefh{Nef_H}
\newop\Eff{Eff}
\newop\Zar{Zar}
\newop\MB{MB}
\newop\MBxC{MB\mathit{(x,C)}}
\newop\NnB{NnB}
\newop\Bigg{Big}
\newop\Effbar{\overline{\Eff}}

\def\keywordname{{\bfseries Keywords}}%
\def\keywords#1{\par\addvspace\medskipamount{\rightskip=0pt plus1cm
\def\and{\ifhmode\unskip\nobreak\fi\ $\cdot$
}\noindent\keywordname\enspace\ignorespaces#1\par}}
\def\subclassname{{\bfseries Mathematics Subject Classification
(2020)}\enspace}
\def\subclass#1{\par\addvspace\medskipamount{\rightskip=0pt plus1cm
\def\and{\ifhmode\unskip\nobreak\fi\ $\cdot$
}\noindent\subclassname\ignorespaces#1\par}}

\begin{document}
\title{On symmetric plane quartic curves}
\author{Marek Janasz}
\date{\today}
\maketitle

\thispagestyle{empty}
\begin{abstract}
In the present paper we study the geometry of plane quartics with large automorphism groups. We show results devoted to smooth plane quartics that are invariant under the action of the elementary abelian group of type $[2,2,2]$, and we study geometric properties of the smooth plane quartic having automorphism group of order $48$.
\subclass{plane curves, line arrangements, singularities}
\end{abstract}
\section{Introduction}
In the present paper we study geometric properties of irreducible symmetric quartic curves $C \subset \mathbb{P}^{2}_{\mathbb{C}}$. Following the lines of \cite{b,m}, we say that an irreducible and reduced curve $C \subset \mathbb{P}^{2}_{\mathbb{C}}$ of genus $g\geq 2$ has \textbf{large automorphism group} if
$$ \# {\rm Aut}(C) > 4(g-1).$$
In the case of plane quartics, the above condition means that $\# {\rm Aut}(C) > 8$. It is worth recalling that smooth plane quartic curves in $\mathbb{P}^{2}_{\mathbb{C}}$ with $\# {\rm Aut}(C) >8$ were classified in \cite{b} with respect to the order of automorphism groups, namely
\begin{itemize}
\item The Klein quartic $C_{168} : x^{3}y + y^{3}z + z^{3}x = 0$ with $\#{\rm Aut}(C_{168}) = 168$,
\item the Dyck quartic $C_{96} : x^{4} + y^{4}+z^{4} = 0$ with $\#{\rm Aut}(C_{96}) = 96$,
\item quartic $C_{48} : x^{4} + y^{4} + xz^{3} = 0$ with $\#{\rm Aut}(C_{48}) = 48$,
\item quartics in one-parameter family $C_{24,a} : x^{4}+y^{4}+z^{4}+3a(x^{2}y^{2} + y^{2}z^{2} + y^{2}z^{2}) = 0$ with $\#{\rm Aut}(C_{24,a}) = 24$, where $a\not\in \{0, (-1\pm \sqrt{-7})/2\}$
\item quartics in one parameter family $C_{16, \delta} : x^{4} + y^{4} + z^{4} + \delta z^{2}y^{2} = 0$ with $\#{\rm Aut}(C_{16,\delta}) = 16$, where $\delta \not\in \{0,\pm 2, \pm 6, \pm 2\sqrt{-3}\}$,
\item quartic $C_{9} : z^{4} + zy^{3} + yx^{3} = 0$ with  $\#{\rm Aut}(C_{9}) = 9$.
\end{itemize}
In the joint work with Pokora and 
Zieli\'nski \cite{JPZ} we studied 
combinatorial properties of line
arrangements determined by bitangents to 
smooth plane quartics having large 
automorphism groups and, in particular, 
we described the weak combinatorics of 
such arrangements. It turns out that the 
arrangements consisting of smooth plane 
quartics and their bitangents allow to 
construct non-trivial free and plus-one 
generated arrangements. In the light of 
these results, we focus on the geometry of 
the smooth plane quartic $C_{48}$ and we 
find effectively all the $28$ bitangent 
lines, the weak combinatorics of the 
arrangement of bitangents, and then we 
construct new examples of free plane 
curve arrangements. In this way we fill 
the gap in our knowledge of arrangements 
consisting of bitangents to symmetric 
plane quartics. The second result of the 
paper is devoted to plane quartics such 
that they are invariant under the action 
of $G_{8} = \mathbb{Z}_{2}^{3}$ and such 
that their Hessian is very symmetric. In 
that setting, our result tells us that if 
$C$ is an irreducible plane quartic which 
is $G_{8}$-invariant having Hessian of 
the form $x^{2}y^{2}z^{2}$, then $C$ is 
given by $Ax^{4}+By^{4}+Cz^{4}$ with 
$ABC=1/1728$.

The structure of the paper goes as follows. In Section $2$, we present our classification result for $G_{8}$-invariant irreducible plane quartics with a prescribed Hessian. In Section $3$, we study the geometry of the plane quartic $C_{48}$ and its $28$ bitangents.
In particular, we provide a classification result on free arrangements consisting of $C_{48}$ and its bitangents.

We work only over the complex numbers, and our computations are supported by \verb}SINGULAR} \cite{Singular}.
\section{$G_{8}$-invariant plane quartic curves}
Here we want to study plane quartics that are invariant under the action of the elementary abelian group of type $[2,2,2]$ of order $8$, and we denote this group by $G_{8}$. In our studies we will use the following faithful matrix representation generated in ${\rm GL}(3,\mathbb{Z})$ by the three matrices:
\begin{align}
\left[\begin{array}{ccc}
 -1  & 0 & 0  \\
  0 & 1 & 0 \\
  0 & 0 & 1
\end{array}\right],
\left[\begin{array}{ccc}
 1  & 0 & 0 \\
 0  & -1 & 0 \\
 0  & 0 & 1
\end{array}\right],
\left[\begin{array}{ccc}
 1  & 0 & 0 \\
 0  & 1 & 0 \\
 0  & 0 & -1
\end{array}\right].
\end{align}
Recall that if $\rho : G \rightarrow {\rm GL}(n,\mathbb{C})$ is a faithful representation of a finite group $G$, then we have the following standard action on the graded ring of polynomials $S:=\mathbb{C}[x,y,z]$, namely
$$G\times S \ni (g,f) \mapsto f(\rho(g)^{-1}\cdot (x,y,z)^{t}) \in S.$$
Since our representation of $G_{8}$ is generated by the diagonal matrices and for every $g \in G_{8}$ one has $\rho(g)^{-1} = \rho(g)$, the action on $S$ can be described as
$$G_{8} \times S \ni (g,f) = f(ax,by,cz) \in S,$$
where $\rho(g) = {\rm Diag}(a,b,c)$ with $a,b,c \in \{-1,1\}$.
Our first observation is the following.
\begin{proposition}
One has $\mathbb{C}[x,y,z]^{G_{8}} = \mathbb{C}[x^{2}, y^{2}, z^{2}]$.
\end{proposition}
\begin{proof}
See for example \cite[p. 57]{MarNeu}.
\end{proof}
From now one we study only these reduced plane quartics $C$ that are $G_{8}$-invariant. Such a quartic has the following defining polynomial
\begin{equation}
\label{genq}
    Q(x,y,z) = Ax^{4} + By^{4} + Cz^{4} + Dx^{2}y^{2} + Ex^{2}z^{2} + F y^{2}z^{2}
\end{equation}
with $A,B,C,D,E,F \in \mathbb{C}$.
In our further investigations, we are going to work with the Hessians associated with plane quartics. In that context, we have the following crucial folkloric result, see \cite[p. 115]{Neusel}.
\begin{proposition}
Let $\rho : G \rightarrow {\rm GL}(n,\mathbb{C})$ be a faithful representation of a finite group $G$. Let $f \in \mathbb{C}[x_{1}, ..., x_{n}]^{G}$ and assume that for every $g \in G$ one has $({\rm det}(\rho(g)))^{2}=1$, then the Hessian of $f$ defined as
$${\rm Hess}(f) = {\rm det}\bigg[\frac{\partial^{2}f}{\partial_{x_{i}}\partial_{x_{j}}}\bigg]_{i,j}$$
is $G$-invariant.
\end{proposition}
In the light of the above result, if we have a $G_{8}$-invariant quartic $C \, : f=0$, then its associated Hessian curve $C_{H} : {\rm Hess}(f) = 0$ is also $G_{8}$-invariant. 

Next, for a general $G_{8}$-invariant quartic given by we compute \eqref{genq} its Hessian, namely
\begin{multline}
{\rm Hess}(Q) = 48ADEx^6 + 48BDFy^{6} + 48CEFz^{6} + \\
(288 ABE + 48ADF - 24D^{2}E)x^{4}y^{2} + (288 ABF + 48 BDE - 24D^{2}F)x^{2}y^{4} \\
(288 BCE + 48CDF - 24EF^{2}) y^{2}z^{4} + (288 BCD + 48 BEF - 24DF^{2})y^{4}z^{2}\\
(288ACD + 48AEF - 24DE^{2})x^{4}z^{2} +  (288ACF + 48CDE - 24E^{2}F)x^{2}z^{4}  \\ 
(1728ABC - 144AF^{2} - 144BE^{2} - 144CD^{2} + 144DEF)x^{2}y^{2}z^{2}.
\end{multline}
Now we are ready to present our result devoted to $G_{8}$-invariant quartics having a very special Hessian.
\begin{theorem}
\label{symq}
Let $C : f=0$ be an irreducible $G_{8}$-invariant quartic in $\mathbb{P}^{2}_{\mathbb{C}}$ such that ${\rm Hess}(f) = x^{2}y^{2}z^{2}$. Then $C$ is given by the equation of the form
$$Q(x,y,z) = Ax^{4} + By^{4} + Cz^{4},$$
where $A, B, C$ such that $ABC=1/1728$.
\end{theorem}
\begin{proof}
Let $C$ be a $G_{8}$-invariant quartic given by \eqref{genq}. Our problem boils down to solving the following system of equations
\begin{equation}
\label{conq}
{\rm Hess}(Q) = x^{2}y^{2}z^{2}.
\end{equation}
The equation \eqref{conq} leads us the following system of equations:
\begin{empheq}[left=\empheqlbrace]{equation}
    \begin{aligned}
     & 48ADE = 0 \\
     & 48BDF = 0 \\
     & 48CEF = 0 \\
     & 288 ABE + 48 ADF - 24D^{2}E = 0 \\
     & 288 ABF + 48 BDE - 24D^{2}F = 0 \\
     & 288 BCE + 48 CDF - 24EF^{2} = 0 \\
     & 288 BCD + 48 BEF - 24DF^{2} = 0 \\
     & 288 ACD + 48 AEF - 24DE^{2} = 0 \\
     & 288 ACF + 48 CDE - 24E^{2}F = 0 \\
     & 1728 ABC - 144 AF^{2} - 144 BE^{2} - 144 CD^{2} + 144 DEF = 1
    \end{aligned}.
\end{empheq}
Using the \verb}SINGULAR} script from the Appendix, we can verify that the above system of equations has the following solutions

\begin{equation*}
\systeme{D=0,E=0,F=0, 1728ABC=1}\, ,
\,\,
\systeme{A=0,B=0,E=0,F=0, 144CD^{2}=-1} \, ,
\,\,
\systeme{A=0,C=0,D=0,F=0, 144BE^{2}=-1} \, ,
\,\,
\systeme{B=0,C=0,D=0,E=0, 144AF^{2}=-1}.
\end{equation*} \\\\
In order to finish our proof, we need to observe that the last three solutions lead to reducible plane quartics. Indeed, if $A=B=E=F=0$ and $144CD^{2}=-1,$
then we get
$$Q(x,y,z) = \frac{-z^{4}}{144D^{2}} + Dx^{2}y^{2} = \frac{-1}{144D^{2}}\bigg(z^{2}+12\sqrt{D^{3}}xy\bigg)\bigg(z^{2} - 12\sqrt{D^{3}}xy\bigg).$$
In the same way, we can show that in the last two remaining cases our quartics are also reducible, and this completes the proof.
\end{proof}
\begin{remark}
Observe that the quartic $C : Ax^{4} + By^{4} + Cz^{4}=0$ is projectively equivalent to the Dyck quartic $C_{96}$ by using the obvious map
$$(x,y,z) \mapsto(ax,by,cz),$$
where $a^{4}=1/A$, $b^{4} = 1/B$, $c^{4}=1/C$.
\end{remark}

\section{On the geometry of the smooth plane quartic with automorphism group of order $48$}
In this section we want to focus on a very particular symmetric smooth plane quartic $C_{48} \subset \mathbb{P}^{2}_{\mathbb{C}}$ with the automorphism group of order $48$ that is given by
\begin{equation}
 F(x,y,z)= x^{4} + y^{4} + xz^{3}.   
\end{equation}
As it is explained in \cite{b}, our quartic curve $C_{48}$ is unique and it has the automorphism group ${\rm Aut}(C_{48}) = C_{4} \odot A_{4}$, i.e., this is the central extension by $C_{4}$ of $A_{4}$, and this group has the following presentation
\begin{multline*}
{\rm Aut}(C_{48}) = \langle a, \, b, \, c, \,  d \,\,| \,\, a^4=d^3=1,b^2=c^2=a^2,ab=ba,ac=ca,ad=da, \\ cbc^{-1}=a^{2}b,dbd^{-1}=a^{2}bc,dcd^{-1}=b\rangle.
\end{multline*}
Our aim here is to understand the geometry of the $28$ bitangents to the quartic $C_{48}$, and in order to do so we will use a geometric description provided by Wall in \cite{w}. First of all, we have the following important description of singular points of the curve dual to the smooth plane quartic.
\begin{proposition}[{\cite[Proposition in Section 1]{w}}]
Let $C \subset \mathbb{P}^{2}_{\mathbb{C}}$ be an irreducible quartic curve and let $C^{\vee}$ be the dual curve to $C$. Then $C^{\vee}$ has the same list of singularities as $C$, except as follows. The nodes, ordinary cusps and triple points of $C$ do not contribute to singularities of $C^{\vee}$. Conversely, $C^{\vee}$ may have nodes, ordinary cusps and singularities of type $E_{6}$ not arising from singularities of $C$.
\end{proposition}
From this perspective, one can ask what we can learn about the geometry of smooth plane quartics by looking at the singularities of the dual curve of degree $12$. We focus here on the situation when the singularities of the dual curves are $E_{6}$ singularities. For such a description, let $\Gamma$ be a curve germ at $P$ with a unique tangent line $P^{\vee}$, $\Gamma^{\vee}$ the dual germ at $P^{\vee}$ with tangent $P$, and denote by $i(\Gamma)$ the local intersection number of $\Gamma$ and $P^{\vee}$.
\begin{proposition}[\cite{w}]
For a smooth plane quartic curve $C \subset \mathbb{P}^{2}_{\mathbb{C}}$, an $E_{6}$ singularity of $C^{\vee}$ dualizes to a smooth germ with a hyperflex tangent, i.e., the local intersection number $i(\Gamma)=4$, and vice versa.
\end{proposition}
In other words, if the dual curve $C^{\vee}$ of a smooth plane quartic $C$ has exactly $k$ singularities of type $E_{6}$, then $C$ admits exactly $k$ hyperosculating lines, i.e. lines tangent to $C$ with $i(\Gamma)=4$. Furthermore, such a tangent point, considered as a singularity created by $C$ and the associated hyperosculating line, is an $A_{7}$ singularity.

Let us come back to $C_{48}$. The dual curve to $C_{48}$, which will be denoted by $C_{48}^{\vee}$, is given by the following polynomial
\begin{multline*}
G(x,y,z)=x^{8}y^{4} + 2x^{4}y^{8} + y^{12} + \frac{256}{27}x^{9}z^{3} + 16x^{5}y^{4}z^{3} + \\ 16xy^{8}z^{3} - 
   \frac{256}{9}x^{6}z^{6} +  64x^{2}y^{4}z^{6} + \frac{256}{9}x^{3}z^{9} - \frac{256}{27}z^{12}.  
\end{multline*}
As we can check directly, using \verb}SINGULAR}, curve $C_{48}^{\vee}$ has exactly $16$ ordinary cusps, $24$ nodes, and $4$ singularities of type $E_{6}$. In the light of the above propositions, our curve $C_{48}$ admits exactly $4$ hyperosculating lines. It is worth recalling here that hyperosculating lines are considered as (degenerate) bitangent lines, so altogether we have $24$ (classical) bitangent lines and $4$ hyperosculating lines to $C_{48}$. 

Here we present equations of the $28$ bitangents to $C_{48}$, namely
{\footnotesize
\begin{longtable}{lp{12cm}}
$\ell_1:$ & $468x+(16r^{13}+78r^9+546r^5-990r)y+(3r^{12}-639)z =0$  \\ 
$\ell_2:$ & $468x-(16r^{13}+78r^9+546r^5-990r)y+(3r^{12}-639)z=0$  \\
$\ell_3:$ & $468x+(50r^{13}+312r^9+2028r^5-198r)y+(3r^{12}-639)z=0$  \\
$\ell_4:$ & $468x-(50r^{13}+312r^9+2028r^5-198r)y+(3r^{12}-639)z=0$  \\
$\ell_5:$ & $468x+(16r^{13}+78r^9+546r^5-990r)y+(30r^{12}+195r^8+1287r^4+45)z=0$  \\ 
$\ell_6:$ & $468x+(16r^{13}+78r^9+546r^5-990r)y+(-33r^{12}-195r^8-1287r^4+594)z=0$  \\ 
$\ell_7:$ & $468x+(50r^{13}+312r^9+2028r^5-198r)y+(30r^{12}+195r^8+1287r^4+45)z=0$  \\ 
$\ell_8:$ & $468x+(-16r^{13}-78r^9-546r^5+990r)y+(-33r^{12}-195r^8-1287r^4+594)=0z$  \\
$\ell_9:$ & $468x+(-16r^{13}-78r^9-546r^5+990r)y+(30r^{12}+195r^8+1287r^4+45)z=0$  \\ 
$\ell_{10}:$ & $468x+(50r^{13}+312r^9+2028r^5-198r)y+(-33r^{12}-195r^8-1287r^4+594)z=0$  \\ 
$\ell_{11}:$ & $468x+(-50r^{13}-312r^9-2028r^5+198r)y+(-33r^{12}-195r^8-1287r^4+594)z=0$  \\ 
$\ell_{12}:$ & $468x+(-50r^{13}-312r^9-2028r^5+198r)y+(30r^{12}+195r^8+1287r^4+45)z=0$  \\ 
$\ell_{13}:$ & $468x+(10r^{13}+78r^9+546r^5+288r)y+(-3r^{12}+171)z=0$  \\  
$\ell_{14}:$ & $468x+(-10r^{13}-78r^9-546r^5-288r)y+(-3r^{12}+171)z=0$  \\ 
$\ell_{15}:$ & $468x+(-10r^{13}-78r^9-546r^5-288r)y+(9r^{12}+39r^8+351r^4-162)z=0$  \\ 
$\ell_{16}:$ & $468x+(-10r^{13}-78r^9-546r^5-288r)y+(-6r^{12}-39r^8-351r^4-9)z=0$  \\ 
$\ell_{17}:$ & $468x+(10r^{13}+78r^9+546r^5+288r)y+(9r^{12}+39r^8+351r^4-162)z=0$  \\ 
$\ell_{18}:$ & $468x+(10r^{13}+78r^9+546r^5+288r)y+(-6r^{12}-39r^8-351r^4-9)z=0$  \\ 
$\ell_{19}:$ & $156x+(-8r^{13}-52r^9-312r^5+144r)y+(-r^{12}+57)z=0$  \\ 
$\ell_{20}:$ & $156x+(8r^{13}+52r^9+312r^5-144r)y+(-r^{12}+57)z=0$ \\ 
$\ell_{21}:$ & $156x+(-8r^{13}-52r^9-312r^5+144r)y+(-2r^{12}-13r^8-117r^4-3)z=0$  \\ 
$\ell_{22}:$ & $156x+(-8r^{13}-52r^9-312r^5+144r)y+(3r^{12}+13r^8+117r^4-54)z=0$  \\ 
$\ell_{23}:$ & $156x+(8r^{13}+52r^9+312r^5-144r)y+(-2r^{12}-13r^8-117r^4-3)z=0$  \\ 
$\ell_{24}:$ & $156x+(8r^{13}+52r^9+312r^5-144r)y+(3r^{12}+13r^8+117r^4-54)z=0$ \\ 
$\ell_{25}:$ & $39x+(2r^{12}+13r^8+78r^4-36)z=0$  \\ 
$\ell_{26}:$ & $39x-(2r^{12}+13r^8+78r^4+3)z=0$  \\
$\ell_{27}:$ & $x=0$ \\
$\ell_{28}:$ & $x+z=0$ 
\end{longtable}
}
where $r^{16} + 6r^{12} + 39r^8 - 18r^4 + 9=0$.  Let us notice that the lines $\ell_{25}, ..., \ell_{28}$ are hyperosculating.

In order to formulate our result devoted to weak-combinatorics of the $28$ bitangent lines we need the following notation. For a given line arrangement $\mathcal{L}$, let us denote by $n_{i} = n_{i}(\mathcal{L})$ the number of $i$-fold intersection points among the lines in $\mathcal{L}$.
\begin{proposition}
The arrangement $\mathcal{L} \subset \mathbb{P}^{2}_{\mathbb{C}}$ consisting of the $28$ bitangents to $C_{48}$ has such intersection points, that:
$$n_{2} = 240, \quad n_{3}=32, \quad n_{4}=7.$$
\end{proposition}
\begin{proof}
This is a standard check that can be done using \verb}SINGULAR}. Let us denote by $f$ be the defining equation of the arrangement $\mathcal{L}$, and denote by $J_{f} = \langle \frac{\partial f}{\partial_{x}}, \frac{\partial f}{\partial_{y}}, \frac{\partial f}{\partial_{z}} \rangle$ the Jacobian ideal. Recall that
$${\rm deg}(J_{f}) = \tau(\mathcal{L}) = n_{2} + 4n_{3} + 9n_{4},$$
where $\tau(\mathcal{L})$ denotes the total Tjurina number of $\mathcal{L}$ (see the precise definition below).
Using \verb}SINGULAR}, we can check that ${\rm deg}(J_{f}) = 431$. Since 
\begin{equation}
\label{naive}
    \binom{28}{2} = n_{2} + 3n_{3} + 6n_{4}
\end{equation}
we obtain 
$$53 = n_{3} + 3n_{4}.$$
Define by $T_{f}$ the ideal generated by all the partial derivatives of order three. Using \verb}SINGULAR} we can check that ${\rm deg}(T_{f}) = 7$, which gives us that
$$n_{3} = 53 - 3n_{4} = 32.$$
Finally, using \eqref{naive} we get $n_{2}=240$, and this completes the proof.
\end{proof}
Now we would like to detect some free arrangements that are determined by the quartic $C_{48}$ and its bitangents.  We need to recall some fundamental definitions.

Let $S := \mathbb{C}[x,y,z]$ be the graded ring of polynomials with complex coefficients, and for a homogeneous polynomial $f \in S$ let $J_{f}$ be the Jacobian ideal given by $f$.

\begin{definition}
Consider the graded $S$-module of Jacobian syzygies of $f$, namely $$AR(f)=\{(a,b,c)\in S^3 : af_x+bf_y+cf_z=0\}.$$
The minimal degree of non-trivial Jacobian relations for $f$ is defined as
$${\rm mdr}(f):=\min\{r : AR(f)_r\neq (0)\}.$$ 
\end{definition}
Recall that for a reduced plane curve $C : f=0$ we denote by $\tau(C)$ its total Tjurina number, i.e., $$\tau(C) = \sum_{p \in {\rm Sing}(C)} \tau_{p},$$
the sum goes over all singular points of $C$ and $\tau_{p}$ denotes the local Tjurina number. Now we can define free plane curves using a result due to du Plessis and Wall \cite{duP}.
\begin{definition}
\label{d2}
Let $C : f=0$ be a reduced curve in $\mathbb{P}^{2}_{\mathbb{C}}$ of degree $d$. Then the curve $C$ with $r:={\rm mdr}(f)\leq (d-1)/2$ is free if and only if
\begin{equation}
\label{duPles}
(d-1)^{2} - r(d-r-1) = \tau(C).
\end{equation}
\end{definition}
\begin{definition}
If $C : f=0$ is a reduced plane curve of degree $d$ in $\mathbb{P}^{2}_{\mathbb{C}}$, then the exponents of $C$ is the pair defined as
$${\rm exp}(C) = ({\rm mdr}(f), d - 1 - {\rm mdr}(f)).$$
\end{definition}
Observe that our arrangements $\mathcal{EL}$ consisting of the quartic curve $C_{48}$ and its bitangents admit only $n_{2}$ nodes, $n_{3}$ ordinary triple points, $n_{4}$ ordinary quadruple points, $t_{3}$ tacnodes, and $t_{7}$ singularities of type $A_{7}$, so
$$\tau(\mathcal{EL}) = n_{2} + 3t_{3} + 4n_{3} + 7t_{7} + 9n_{4}.$$

\begin{theorem}
The quartic curve $C_{48}$ and its bitangents admits exactly one free arrangement of degree $8$, namely the arrangement admits one ordinary quadruple point and four singularities of type $A_{7}$.
\end{theorem}
\begin{proof}
Assume that $\mathcal{EL}$ given by $f=0$ is a free arrangement consisting of the quartic $C_{48}$ and $4$ bitangents. We want to find some numerical constraints on such free arrangements. First of all, recall that the arrangement $\mathcal{EL}$ has to satisfy the following system of Diophantine equations:
\begin{equation*}(\square) : \quad 
\systeme{n_{2} + 2t_{3} + 3n_{3} + 4t_{7} + 6n_{4} = 4\cdot 4 + \binom{4}{2} = 22, n_{2}+3t_{3} + 4n_{3} + 7t_{7}+9n_{4} = \tau(\mathcal{EL})},
\end{equation*}
wherein $\tau(\mathcal{EL}) \in \{37,39\}$. Indeed, the first equation follows from B\'ezout's Theorem, and for the second equation we need to observe that by \cite[Theorem 2.1]{DimcaSernesi} and the fact that ${\rm mdr}(f) \leq (d-1)/2$ one has
$$\frac{7}{2} \geq {\rm mdr}(f) \geq \frac{1}{2}\cdot 8-2 = 2,$$
which implies that ${\rm mdr}(f) \in \{2,3\}$. Using \eqref{duPles}, we get $\tau(\mathcal{EL}) \in \{37,39\}$. It turns out that for $\tau(\mathcal{EL})=39$ our system $(\square)$ does not have any non-negative integer solution, and for $\tau(\mathcal{EL})=37$ we have exactly three solutions, namely
$$(n_{2},n_{3},n_{4},t_{3},t_{7}) \in \{(0,0,0,3,4),(2,0,0,0,5),(0,0,1,0,4)\}.$$
Observe that the first two weak-combinatorics cannot be realized geometrically using bitangents and quartic $C_{48}$, i.e., the first one cannot be realized due to B\'ezout's theorem (too many intersections), and the second combinatorics cannot be realized since having $4$ bitangents we can produce at most $4$ singularities of type $A_{7}$. The last combinatorics can be realized geometrically. Consider the following arrangement $\mathcal{EL}$ given by
$$Q(x,y,z) = (x^{4}+y^{4}+xz^{3})\cdot\ell_{25}\cdot\ell_{26}\cdot\ell_{27}\cdot\ell_{28}.$$
We can check directly that the lines $\ell_{25}, \ell_{26}, \ell_{27}, \ell_{28}$ are concurrent, which means that they intersect at a quadruple points. Moreover, each line $\ell_{i}$ is tangent to quartic $C_{48}$ at exactly one point, so the local intersection index is equal to $4$, and we have a singular point of type $A_{7}$. Summing up, we have $n_{4}=1$ and $t_{7}=4$. We can check using \verb}SINGULAR} that ${\rm mdr}(Q)=3$, and
$$37 = r^{2} -r(d-1)+(d-1)^{2} = \tau(\mathcal{EL}) = 7t_{7}+9n_{4} = 37,$$
hence $\mathcal{EL}$ is free with exponents $(3,4)$. Since in our arrangement of bitangents we have exactly $4$ hyperosculating lines, our arrangement is unique.
\end{proof}
\begin{remark}
In light of \cite[Theorem 1.1]{dim2}, if we delete one (and any) line from $\mathcal{EL}$, we get either a free or plus-one generated arrangement, and it turns out that we end up with the second scenario, namely we get a plus-one generated curve.
\end{remark}
\section*{Acknowledgement}
I would like to thank Xavier Roulleau for helpful comments and \verb}MAGMA} computations.

Marek Janasz is supported by the National Science Centre (Poland) Sonata Bis Grant  \textbf{2023/50/E/ST1/00025}. For the purpose of Open Access, the authors have applied a CC-BY public copyright licence to any Author Accepted Manuscript (AAM) version arising from this submission.
\section*{Appendix}
Here we present our \verb}SINGULAR} script that allows to verify the claim in Theorem \ref{symq}.

\begin{center}
\begin{BVerbatim} 

option(redSB);
LIB "primdec.lib";
 
ring R = 0, (A,B,C,D,E,F), dp;
ideal si = 
A*D*E,
B*D*F,
C*E*F,
12*A*B*E + 2*A*D*F - D2*E,
12*A*B*F + 2*B*D*E - D2*F,
12*B*C*E + 2*C*D*F - E*F2,
12*B*C*D + 2*B*E*F - D*F2,
12*A*C*D + 2*A*E*F - D*E2,
12*A*C*F + 2*C*D*E - E2*F,
1728*A*B*C - 144*A*F2 - 144*B*E2 - 144*C*D2 + 144*D*E*F - 1;
primdecGTZ(si);
\end{BVerbatim}
\end{center}

\vskip 0.5 cm
\bigskip
Marek Janasz,
Department of Mathematics,
University of the National Education Commission Krakow,
Podchor\c a\.zych 2,
PL-30-084 Krak\'ow, Poland. \\
\nopagebreak
\textit{E-mail address:} \texttt{marek.janasz@up.krakow.pl}

\end{document}